\documentclass{amsart}
\usepackage{amsmath}
\usepackage[toc,page]{appendix}

 \usepackage{tikz}
\usepackage{pgfplots}
\usepackage{tikz-3dplot}
\usetikzlibrary{quotes,angles,positioning}
 
\tdplotsetmaincoords{60}{115}
\pgfplotsset{compat=newest}
 \tdplotsetmaincoords{70}{65}

\usepackage{amssymb,xcolor,url}
\usepackage[many]{tcolorbox}

 \newtheorem{thm}{Theorem}[section] 
 
 \newtheorem{lem}[thm]{Lemma} 
 \newtheorem{prop}[thm]{Proposition}
 \newtheorem{cor}[thm]{Corollary}
 
 \theoremstyle{remark}
 
 \theoremstyle{remark}
 \newtheorem{exa}[thm]{Example} 
\theoremstyle{definition}
\newtheorem{defin}[thm]{Definition} 

 \title[The 4-player gambler's ruin problem]{The 4-player gambler's ruin problem}
\author{Kathryn O'Connor}
\address{Department of Mathematics, Cornell University}
\email{kfo23@cornell.edu}
\thanks{K. O'Connor's research was partially supported by NSF grant DMS-1645643}

\author{Laurent Saloff-Coste}
\address{Department of Mathematics, Cornell University}
\email{lps2@cornell.edu}
\thanks{L. Saloff-Coste's research was partially supported by NSF grant DMS-2054593}
\begin{document}

\begin{abstract}
This work explains how to utilize earlier results by P. Diaconis, K. Houston-Edwards and the second author to estimate probabilities related to the 4-player gambler ruin problem. For instance, we show that the probability that a very dominant player (i.e., a player starting with all but 3 chips distributed among the remaining players) is first to loose is of order $N^{-\alpha}$ where $\alpha$ is approximately $5.68$. In the $3$-player game, this probability is or order $N^{-3}$. We note it is futile to attempt to give heuristic/intuitive explanations for the value of $\alpha$. This value is obtained via an explicit formula relating $\alpha$ to the Dirichlet eigenvalue $\lambda$ (zero boundary condition) of the spherical Laplacian in the equilateral spherical triangle on the unit sphere $\mathbb S^2$ that corresponds to a unit simplex with one vertex placed at the origin in Euclidean $3$-space. The value of $\lambda$ is estimated using a finite-difference-type algorithm developed by Grady Wright.

\end{abstract}

\maketitle

\section{introduction}  Consider the following multiplayer version of the classical gambler's ruin problem.  There are $k$ players and a fixed total amount of chips, $N$, distributed between them with player $i$ holding $x_i$ chips. At each turn, a pair $ij=\{i,j\}$, $1\le i<j\le k$, is picked uniformly at random and the chosen two players play a fair Heads-or-Tails game and exchange one chip as a result. We consider this game until the time $\tau$ when one of the $k$ players is left with no chips. 
Let $\mathbf s=(s_1,\dots,s_k)$ be the distribution of chips at the start of the game. Let $\mathbf X^n=(X^n_1,\dots,X^n_k)$
be the distribution of the chips after $n$ games, $n\le \tau$.  

For a review of this problem and the associated literature, the reader is referred to  \cite{DE21} and \cite{DHESC2}. Here we only recall that a (continuous) version of this problem is discussed in \cite{Cover1987}. The aim of this article is to complement \cite{DE21,DHESC2} with a careful discussion of the application of the results of \cite{DHESC2} to the 4-player version in a spirit similar to  the treatment of the  3-player version in \cite{DE21}.   The general  goal is to provide an informative analytic description of the large $N$ behavior of 
$$\mathbf P(\mathbf  X_\tau=\mathbf z| \mathbf X_0=\mathbf s),\; \mathbf s\in \mathfrak S_N, \; \mathbf z\in \overline{\mathfrak S}_N\setminus \mathfrak S_N,$$
uniformly over all $\mathbf s,\mathbf z$, where
$$\overline{\mathfrak S}_N=\left\{(x_i)_1^k: 0\le x_i, \sum_1^kx_i=N\right\},\;\;\mathfrak S_N=\left\{(x_i)_1^k: 0< x_i, \sum_1^kx_i=N\right\}.$$ 
Here, all points have integer coordinates. See Figure \ref{fig1} and Figure \ref{fig2}.
Note that the point of  $\overline{\mathfrak S}_N\setminus \mathfrak S_N$ which have more than one $0$ coordinates cannot actually be reached  by $(X_n)_{0\le n\le \tau}$. For this reason, it is convenient to introduce 
$$\partial \mathfrak S_N=\left\{\mathbf z=(z_1,\dots,z_k)\in \overline{\mathfrak S_N}\setminus \mathfrak S_N: \mbox{ exactly one } z_i=0, 1\le i\le k\right\}$$
and $$\mathfrak T_{N,i}=\partial \mathfrak S_N \cap \{x_i=0\}.$$

\begin{figure}
    \centering

\begin{tikzpicture}[tdplot_main_coords, scale=4]

\coordinate (O) at (0,0,0);
\coordinate (X) at (1.2,0,0);
\coordinate (Y) at (0,1.2,0);
\coordinate (Z) at (0,0,1.2);

\draw[thick,->] (0,0,0) -- (1.8,0,0) node[anchor=north east]{$x_1$};
\node at (X) [below left = 0.3mm of X] {$(N,0,0)$};
\draw[thick,->] (0,0,0) -- (0,1.7,0) node[anchor=north west]{$x_2$};
\node at (Y) [ below right = 0.3mm of Y] {$(0,N,0)$};
\draw[thick,->] (0,0,0) -- (0,0,1.6) node[anchor=south]{$x_3$};
\node at (Z) [above left = 0.3mm of Z] {$(0,0,N)$};

\draw (X) -- (Y);
\draw (Y) -- (Z);
\draw (Z) -- (X);

\fill[tdplot_main_coords, opacity=0.1] (X) -- (Y) -- (Z) -- cycle;

{
\draw[opacity=0.6] (0.2,1,0) -- (0,1,0.2);
\draw[opacity=0.6] (0.4,0.8,0) -- (0,0.8,0.4);
\draw[opacity=0.6] (0.6,0.6,0) -- (0, 0.6,0.6);
\draw[opacity=0.6] (0.8,0.4,0) -- (0,0.4,0.8);
\draw[opacity=0.6] (1,0.2,0) -- (0,0.2,1);

\draw[opacity=0.6] (0.2,1,0) -- (0.2,0,1);
\draw[opacity=0.6] (0.4,0.8,0) -- (0.4,0,0.8);
\draw[opacity=0.6] (0.6,0.6,0) -- (0.6,0,0.6);
\draw[opacity=0.6] (0.8,0.4,0) -- (0.8,0,0.4);
\draw[opacity=0.6] (1,0.2,0) -- (1,0,0.2);

\draw[opacity=0.6] (0,0.2,1) -- (0.2,0,1);
\draw[opacity=0.6] (0,0.4,0.8) -- (0.4,0,0.8);
\draw[opacity=0.6] (0, 0.6,0.6) -- (0.6,0,0.6);
\draw[opacity=0.6] (0,0.8,0.4) -- (0.8,0,0.4);
\draw[opacity=0.6] (0,1,0.2) -- (1,0,0.2);
}

\end{tikzpicture}

\caption{The triangle for the 3-player game in the plane $\sum_1^3x_i=N$. Here $N=6$.}
    \label{fig0}
\end{figure}
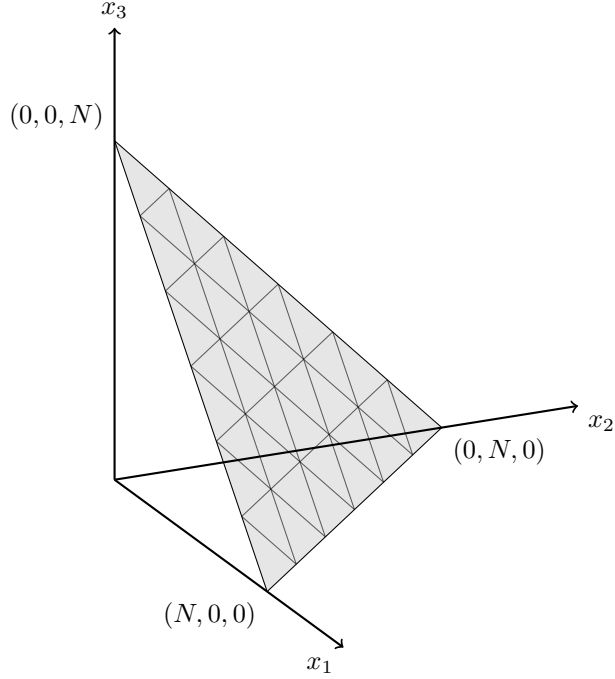

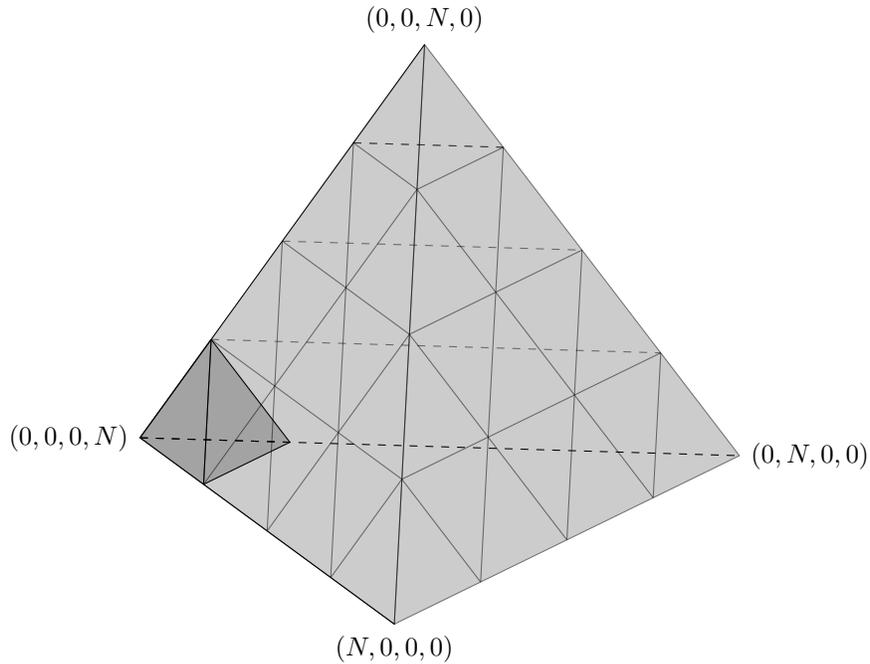
\begin{figure}
    \centering
\begin{tikzpicture}[tdplot_main_coords, scale=2]

\coordinate (O) at (0,0,0);
\coordinate (A) at (1,0,0);
\coordinate (B) at (0.5, 0.866025, 0);
\coordinate (C) at (0.5, 0.288675, 0.8164966);

\coordinate (A2) at (4, 0,0);
\node at (A2) [below = 0.3mm of A2] {$(N,0,0,0)$};
\coordinate (B2) at (2, 4*0.866025 , 0);
\node at (B2) [ right = 0.3mm of B2] {$(0,N,0,0)$};
\coordinate (C2) at (2, 4*0.288675 , 4* 0.8164966 );
\node at (C2) [above = 0.3mm of C2] {$(0,0,N,0)$};

\node at (O) [left = 0.3mm of O] {$(0,0,0,N)$};


\draw[dashed] (O) -- (B);
\draw (0,0,0) -- (C);
\draw (A) -- (B);
\draw (B) -- (C);
\draw (A) -- (C);

\draw[dashed] (O) -- (4*0.5, 4*0.866025 , 0);
\draw (O) -- (4*0.5, 4*0.288675 , 4* 0.8164966 );
\draw (O) -- (4*1, 0,0);

\fill[opacity=0.2] (O) -- (A) -- (B) -- (C) -- cycle;
\fill[opacity=0.2] (O) -- (A2) -- (B2) -- (C2) -- cycle;


\foreach \x in {0.25, 0.5, 0.75, 1}
\foreach \y in {4}
{\draw[opacity=0.5] (\y - \y*\x + \y*\x*0.5,  \y*\x*0.866025, 0) 
-- (\y - \y*\x  + \y*\x*0.5,  \y*\x*0.288675, \y*\x*0.8164966 ); 
\draw[opacity=0.5] (\y*0.5 - \y*\x*0.5 + \y*\x ,  \y*0.866025 - \y*\x*0.866025, 0  ) 
-- (\y*0.5 - \y*\x*0.5 + \y*\x*0.5,  \y*0.866025 - \y*\x*0.866025 + \y*\x*0.288675 , \y*\x* 0.8164966  ); 
\draw[opacity=0.5] (\y*0.5 - \y*\x*0.5 + \y*\x ,  \y*0.288675 - \y*\x*0.288675 , \y*0.8164966 - \y*\x*0.8164966 ) 
-- (\y*0.5 - \y*\x*0.5 + \y*\x*0.5 ,  \y*0.288675 - \y*\x*0.288675 + \y*\x*0.866025 , \y* 0.8164966 - \y*\x* 0.8164966 ); 
}

\foreach \x in {1,2,3,4}
{
\draw[opacity=0.5] (\x,0,0) -- (\x + 4*0.5 - \x*0.5, 4*0.288675 - \x*0.288675, 4*0.8164966 - \x*0.8164966);
\draw[opacity=0.5] (4*0.5 - \x * 0.5, 4*0.288675 - \x * 0.288675, 4*0.8164966 - \x*0.8164966) -- (\x + 4*0.5 - \x*0.5, 4*0.288675 - \x*0.288675, 4*0.8164966 - \x*0.8164966);
\draw[opacity=0.5] (\x,0,0) -- (\x * 0.5, \x * 0.288675,  \x*0.8164966);
}

\foreach \x in {1,2,3}
{
\draw[dashed, opacity=0.5] (\x*0.5 , \x*0.288675  , \x* 0.8164966 ) -- (4* 0.5 ,4* 0.866025 - \x*0.866025 + \x*0.288675 , \x*0.8164966);
}

\coordinate (topO) at (3*0.5, 3*0.288675, 3*0.8164966);
\coordinate (topA) at (1+3*0.5, 3*0.288675, 3*0.8164966);
\coordinate (topB) at (0.5 + 3*0.5, 0.866025 + 3*0.288675, 3*0.8164966);

\draw[dashed, opacity=0.5] (topO) -- (topB);


\end{tikzpicture}
  
    \caption{The 3-dimensional 4-player tetrahedron (N=4). The picture shows only the vertices that are on the two visible front faces and dashed horizontal lines on the back face. In 4-space, the left-most corner
    hangs from the $x_4$-axis and the right-most face lies in the $\{x_4=0\}\cap \{\sum_1^4x_i=N\}$ plane. The entire tetrahedron lies in the $\{\sum_1^4 x_i=N\}$ 3-space. Compare to the 3-player game picture shown in Figure \ref{fig0}.}
    \label{fig1}
\end{figure}
 
Among the particular questions one can ask, one test question identified in \cite{DE21} concerns the behavior of
$$\mathbf P\left(\mathbf  X_\tau\in \{\mathbf z: z_k=0\}  | \mathbf X_0=\mathbf s\right) \mbox{ and }
\mathbf P\left(\mathbf  X_\tau =\mathbf z  | \mathbf X_0=\mathbf s \mbox{ and } \mathbf  X_\tau\in \{\mathbf z: z_k=0\} \right)
$$when $\mathbf s=(s_i)_1^k$ satisfies $s_k=N-k+1$.  In words, what is the probability that a very dominant player at time $0$ (i.e., a player with  $N-k+1$ chips) be the first to loose all their chips? Given that the very dominant player at time $0$ is the first to loose, what is the probability distribution describing the allocation of chips among the remaining players at the time the former dominant player looses?

\begin{thm} \label{th1}Referring to the 4-player game with player $4$ being the very dominant player at the start of the game ($s_4=N-3, s_1,s_2,s_3=1$, $\mathbf s^*_N=(1,1,1,N-3)$), 
there exists $\alpha>0$ (the approximate value of $\alpha$ is $\alpha\approx 5.68$)such that
$$\mathbf P(\mathbf  X_\tau\in \mathfrak T_{N,4}  | \mathbf X_0=\mathbf s_N^*) \asymp N^{-\alpha}$$
and, for $\mathbf z=(z_1,z_2,z_3,0)\in \mathfrak T_{N,4}$ and  $\beta  =\pi/\arccos (1/3)\approx 2.55$,
$$\mathbf P(\mathbf  X_\tau=\mathbf z  | \mathbf X_0=\mathbf s \mbox{ and } \mathbf  X_\tau\in \mathfrak T_{N,4})
\asymp \frac{z_1z_2z_3[(z_1+z_2)(z_1+z_3)(z_2+z_3)]^{\beta-2}}{N^{3\beta-1}}  .$$
\end{thm}

For simplicity, Theorem \ref{th1} addresses only a very extreme situation and it is desirable to treat other cases as well as will be done later in Section \ref{sec-Appli}. The given value of $\alpha$ is from a numerical computation which will be briefly discussed at the end of the paper. Note that there is no simple heuristic allowing one to guess correctly the behavior of 
$$\mathbf P(\mathbf  X_\tau\in \mathfrak T_{N,4}  | \mathbf X_0=\mathbf s_N^*).$$

{\bf Acknowledgements} We thank Grady Wright from Boise State University for providing an algorithm that computes the approximate value of $\alpha$ given in the above theorem (more on that later). We thank Alex Townsend for paying attention to our questions and connecting us with Grady. And we thank Persi Diaconis for encouraging us in our efforts.

\section{Key estimates in terms of the Perron-Frobenius eigenfunction}
\subsection{Notation} Throughout, we always assume that the integers $N,k$ are such that $3\le k\le N$. In fact, we always assume that $k$ is fixed and we are interested in what happens for that fixed $k$ when $N$ is large.
The process $(\mathbf X^n)_{n\ge 0}$ is a Markov chain taking place in the simplex $$\overline{\mathfrak S}_N=\left\{(x_i)_1^k\in \mathbb Z^k: 0\le x_i, \sum_1^kx_i=N\right\},$$ started at $\mathbf s$ in
$$\mathfrak S_N=\left\{(x_i)_1^k\in \mathbb Z^k: 0< x_i, \sum_1^kx_i=N\right\}$$ and stopped at the random time $\tau$ when it hits the boundary
$\partial \mathfrak S_N\subset \overline{\mathfrak S}_N\setminus \mathfrak S_N$ defined earlier. This Markov process is based on the global Markov kernel
$$\widetilde{K}(\mathbf x,\mathbf y)=\left\{\begin{array}{cl}\frac{1}{2{k\choose 2}} &\mbox{ if }   \mathbf y -\mathbf x\in\{\pm( e_i-e_j): 1\le i<j\le k\}\\
0& \mbox{ otherwise,} \end{array} \right.$$
in the hyperplane $\left\{\mathbf x=(x_i)_1^k: \sum_1^nx_i=N\right\}$, and its killed version, $K=\widetilde{K}_{\mathfrak S_N}$ defined by
\begin{equation}\label{defK}K(\mathbf x,\mathbf y)=\widetilde{K}(\mathbf x,\mathbf y)\mathbf 1_{\mathfrak S_N}(\mathbf x)\mathbf 1_{\mathfrak S_N}(\mathbf y).\end{equation}
This kernel defines an operator (by abuse of notation, we call it $K$) acting on functions on the finite set $\mathfrak S_N$ by
$$K\phi(\mathbf x)=\sum_{\mathbf y\in \mathfrak S_N} K(\mathbf x,\mathbf y)\phi(\mathbf y).$$
We consider this operator acting on the Hilbert space $L^2(\mathfrak S_N)$ equipped with
$$\langle \phi,\psi\rangle=\sum_{\mathbf x\in\mathfrak S_N} \phi(\mathbf x)\psi(\mathbf x),\;\;\; \|\phi\|_2^2=\langle \phi,\phi\rangle.$$

The sub-Markovian kernel $K$ is irreducible in $\mathfrak S_N$ in the sense that there exists $m$ such that $K^m(\mathbf x,\mathbf y)>0$ for all $\mathbf x,\mathbf y\in \mathfrak S_N$. It  thus follows from the Perron-Frobenius theorem that there is a unique $\phi_0:\mathfrak S_N \to (0,+\infty)$ such that
\begin{equation}\label{defphi0} K\phi_0=\beta_0\phi_0
\;\;\|\phi_0\|_2=1,\;\mbox{ and } \beta_0=\sup\{ \|K\phi\|_2: \|\phi\|_2=1\}.\end{equation}

Note that, at the exit time $\tau$, the process $(X_n)_{0\le n\le \tau}$ started at 
$X_0\in \mathfrak S_N$, is at point in $\partial \mathfrak S_N$ for the first time.  For any point $\mathbf z\in \partial \mathfrak S_N$, let $\mathcal N_\mathbf z$ be the non-empty finite set of those points in $\mathfrak S_N$
which are neighbors of $z$ in the sense that $\widetilde{K}(\mathbf y,\mathbf z)>0$
if $\mathbf y\in \mathcal N_\mathbf z$. Namely,
$$\mathcal N_\mathbf z=\{\mathbf y\in \mathfrak S_N: \widetilde{K}(\mathbf y,\mathbf z)>0\}, \;\;\mathbf z\in \partial \mathfrak S_N.$$
We will often use the notation $\mathbf y_z$ to denote an arbitrary point in $\mathcal N_\mathbf z$.

\subsection{Estimates in terms of $\phi_0$}

According to \cite{DHESC1,DHESC2},what is  needed to answer questions such as the ones considered in Theorem \ref{th1}  is estimates describing the boundary behavior of the Perron-Frobenius function $\phi_0$.
Indeed, for $k\ge 3$, one can apply \cite[(6.15) and Theorems 6.5-6.16]{DHESC2} and these results give the following very general estimates (by ``very general,'' we mean that these estimates holds in much greater generality than the present  setting of the $k$-player game. See \cite{DHESC1,DHESC2}.)
In all these estimates, the dimension $k$ is fixed and all implied constants may depend on $k$. The key point is that these estimates are uniform in $N$.  The first Lemma gives basic size estimate for $\phi_0$ in the middle of $\mathfrak S_N$ and captures the fact that $\phi_0$ behaves in a very tame fashion in the middle of $\mathfrak S_N$. This is the result of a basic Harnack type inequality that applies to $\phi_0$. In these statements, we can choose to use the Euclidean distance to compute distance between points, or the graph distance in the lattice supporting $\mathfrak S_N$ in the hyperplane $\{\sum x_i=N\}$. The various implied constants will then have to be adjusted depending of the choice, and these adjustments depend on the dimension $k$. 

\begin{lem} \label{lem-phi1}
 The normalized Perron-Frobenius eigenvalue $ \beta_0$ and eigenfunction $\phi_0$ satisfy
 $$1-\beta_0\asymp N^{-2}$$
 and, with $\mathbf o_N$ denoting any one of the  points in $\mathfrak S_N$ closest to the middle point $(N/k,\cdots,N/k)$, 
 $$\phi_0(\mathbf o_N)\asymp N^{-(k-1)/2},\;\; \sum_{\mathbf s\in \mathfrak S_N} \phi_0(\mathbf s)\asymp N^{(k-1)/2}.$$
 Moreover, for any $a\in (0,1)$ there is a constant $A$ for which
 $$A^{-1}\phi(\mathbf o_N)\le \phi_0(\mathbf x)\le A\phi_0(\mathbf o_N) $$
 for all $\mathbf x\in \mathfrak S_N$ such that $d(\mathbf x,\partial \mathfrak S_N)\ge aN$.
 \end{lem}
The estimate of $\phi_0$ given in this lemma is for points that are away from the boundary $\partial \mathfrak S_N$. In that region, $\phi_0$ behave like a constant. Even though this lemma does not capture this fact, it is also true that $\phi_0(\mathbf o_N)\asymp \|\phi_0\|_
\infty$. See, e.g., \cite[Theorem 6.6]{DHESC2} and \cite[Theorem 8.9]{DHESC1}

\begin{thm}[{\cite[(6.22)]{DHESC2}}] \label{th2}Fix $k\ge 3$. Referring to the k-player game,
consider arbitrary points $\mathbf s\in \mathfrak S_N$ and $\mathbf z\in \partial \mathfrak S_N$. We have
$$\mathbf P(\mathbf  X_\tau =\mathbf z  | \mathbf X_0=\mathbf s) \asymp \phi_0(\mathbf s) \phi_0(\mathbf y_z) \left(N^2
+ \frac{1}{\phi_0(\mathbf s_d)^2d^{k-3}}\right)$$
where $\mathbf y_\mathbf z\in \mathcal N_\mathbf z$, $d=d(\mathbf s,\mathbf z)$ and $\mathbf s_d$ is any point in $\mathfrak S_N$ such that $$d(\mathbf s,\mathbf s_d)\le C_k d \mbox{ and  } d(\mathbf s_d,\partial \mathfrak S_N)\ge c_k d$$ for some appropriately chosen constants $0< c_k\le C_k<+\infty$.
\end{thm}

Given Lemma \ref{lem-phi1}, Theorem \ref{th1} gives
\begin{equation} \label{P1}
  \mathbf P(\mathbf  X_\tau =\mathbf z  | \mathbf X_0=\mathbf s) \asymp_{\epsilon,k} N^2\phi_0(\mathbf s) \phi_0(\mathbf y_z) 
\end{equation}
for each fixed $\epsilon\in (0,1)$ and  all  points $\mathbf s\in \mathfrak S_N$ and $\mathbf z\in \partial \mathfrak S_N$ with $d(\mathbf s,\mathbf z)> \epsilon N$. This is because
$d=d(\mathbf s,\mathbf z)\asymp N$ and thus
$\phi_0(\mathbf x_d)^2d^{k-3}\asymp N^{-2}$.

Estimate (\ref{P1}) is what is needed to obtain Theorem \ref{th1} because, in that theorem, the starting point
$\mathbf s_N=(1,1,1,N-3)$ is at distance of order $N$ of the face $\mathfrak T_{N,4}=\partial \mathfrak S_N\cap \{z_4=0\}$. See, e.g., Figures \ref{fig1} and \ref{fig2}.

In all other cases, that is, whenever $d(\mathbf s,\mathbf z)\le \epsilon N$, we have
\begin{equation}\label{P2}
  \mathbf P(\mathbf  X_\tau =\mathbf z  | \mathbf X_0=\mathbf s) \asymp_{\epsilon,k} \frac{\phi_0(\mathbf s) \phi_0(\mathbf y_z)}{\phi_0(\mathbf s_d)^2d^{k-3}} ,\;\;d=d(\mathbf s,\mathbf z).
\end{equation}

The results in \cite{DHESC2} also provides  detailed information in the same spirit  for hitting probabilities under the additional requirement that $\tau<t$.
This will be discussed in a later section.

\section{Estimating $\phi_0$}
Theorem \ref{th2} makes it clear that having a detailed understanding of $\phi_0$ is what is needed to obtain good hitting probability estimates. We will give explicit estimates for the case $k=3,4$ and $5$.  In the case $k=3$,
\cite{DHESC2} gives an explicit exact formula for $\phi_0$ as well as the two sided estimate (\cite[(5.20)]{DHESC2})
$$\phi_0((s_1,s_2,s_3))
\asymp  N^{-7}(s_1+s_2)(s_1+s_3)(s_2+s_3)s_1s_2s_3.$$
Unfortunately, in the case $k>3$, there is no reasons to expect exact formulas expressible in simple terms and we need to rely on a much more sophisticated analysis.

\begin{figure}
 \centering
\begin{tikzpicture}[tdplot_main_coords, scale=2]



\coordinate (O) at (0,0,0);
\coordinate (A) at (0.4,0,0);
\coordinate (B) at (0.4*0.5, 0.4*0.866025, 0);
\coordinate (C) at (0.4*0.5, 0.4*0.288675, 0.4*0.8164966);

\coordinate (A2) at (4, 0,0);
\coordinate (B2) at (2, 4*0.866025 , 0);
\coordinate (C2) at (2, 4*0.288675 , 4* 0.8164966 );

\fill[opacity=0.2] (O) -- (A) -- (B) -- (C) -- cycle;
\draw[dashed] (O) -- (B);
\draw (0,0,0) -- (C);
\draw (A) -- (B);
\draw (B) -- (C);
\draw (A) -- (C);


\draw[dashed] (O) -- (2.5, 5*0.866025 , 0);
\draw (O) -- (2.5, 5*0.288675 , 5* 0.8164966 );
\draw (O) -- (5, 0,0);

\fill[opacity=0.2] (O) -- (A2) -- (B2) -- (C2) -- cycle;


\foreach \x in {0.1, 0.2, 0.3, 0.4, 0.5, 0.6, 0.7, 0.8, 0.9, 1}
\foreach \y in {4}
{\draw[opacity=0.5] (\y - \y*\x + \y*\x*0.5,  \y*\x*0.866025, 0) 
-- (\y - \y*\x  + \y*\x*0.5,  \y*\x*0.288675, \y*\x*0.8164966 ); 
\draw[opacity=0.5] (\y*0.5 - \y*\x*0.5 + \y*\x ,  \y*0.866025 - \y*\x*0.866025, 0  ) 
-- (\y*0.5 - \y*\x*0.5 + \y*\x*0.5,  \y*0.866025 - \y*\x*0.866025 + \y*\x*0.288675 , \y*\x* 0.8164966  ); 
\draw[opacity=0.5] (\y*0.5 - \y*\x*0.5 + \y*\x ,  \y*0.288675 - \y*\x*0.288675 , \y*0.8164966 - \y*\x*0.8164966 ) 
-- (\y*0.5 - \y*\x*0.5 + \y*\x*0.5 ,  \y*0.288675 - \y*\x*0.288675 + \y*\x*0.866025 , \y* 0.8164966 - \y*\x* 0.8164966 ); 
}


\node at (O) [left = 0.3mm of O] {$(0,0,0,N)$};

\end{tikzpicture}
\caption{A larger 3-dimensional 4-player tetrahedron (N=10), part of a lattice cone (not shown but easily imagined) and sitting in a continuous cone. The edges of the small dark grey tetrahedron indicate the basic steps of the lattice. Attention: contrary to what happens in the 2-dimensional case,  the large tetrahedron is NOT paved by isometric copies of the small tetrahedron.}
    \label{fig2}
\end{figure}
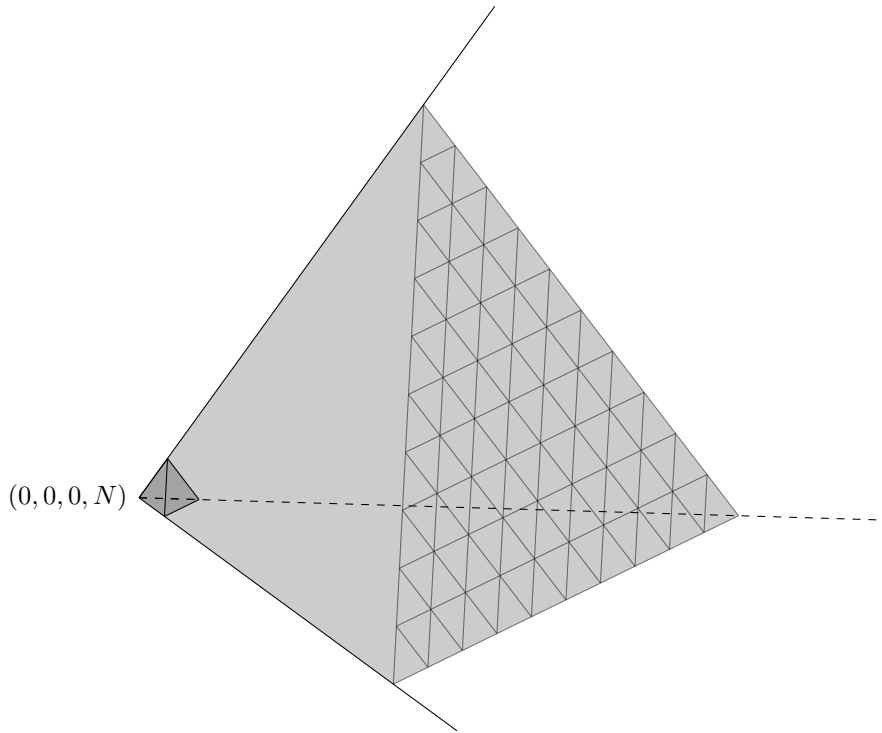

\subsection{From the simplex to the lattice cone}

The simplex $\mathfrak S_N$ lies in the lattice $\mathfrak L$ generated by
the vectors  $(e_i-e_j)$, $1\le i\neq j\le k$, in the hyperplane 
$\{\mathbf x=(x_1,\dots,x_k): \sum_1^k x_i=N\}$. In that hyperplane, we can consider the continuous open half-cone (here $\mathbf x\in \mathbb R^k$)
$$\mathcal C=\left\{\mathbf x= (t\xi_1,\dots, t\xi_{k-1}, N(1-t)): t> 0, \xi_i>0, 1\le i\le k-1, \sum_1^{k-1} \xi_i=N\right\}$$
and 
the lattice cone
$$\mathfrak C_\mathfrak L=\mathcal C\cap \left\{\mathbf x=(x_1,\dots,x_k)\in \mathbb Z^k: \mathbf x=\sum_{1\le i<j\le k} m_{ij}(e_i-e_j), m_{ij}\in\mathbb Z\right\}.$$ 
Note that the tip of this cone is the point $\mathbf t_N$ with coordinates $$t_1=\dots=t_{k-1}=0,\;t_k=N.$$ 
See Figure \ref{fig2}. Note also that this description of the cone is NOT a description in polar coordinate, that is, the point $(\xi_1,\dots,\xi_{k-1})$
is not on a sphere centered at $\mathbf t_N$ in the $\{\mathbf x: \sum_1^N x_i=N\}$. It is in the $(k-2)$-dimensional face of the simplex that lies in $\{x_k=0\}$.

Now, in addition of the Perron-Frobenius eigenfunction $\phi_0$ defined in $\mathfrak S_N$, consider a function $u$ defined on the lattice  $\mathfrak L $ satisfying
\begin{equation}\label{coneprofile}
\left\{\begin{array}{l}\sum_{\mathbf y}\widetilde{K}(\mathbf x,\mathbf y)u(\mathbf y)=u(\mathbf x) \mbox{ for } \mathbf x\in \mathfrak C_\mathfrak L\\
u(\mathbf x)=0 \mbox{ when } \mathbf x\in\mathcal \mathfrak L\setminus \mathfrak C_\mathfrak L \mbox{ and } u>0 \mbox{ in }\mathfrak C_\mathfrak L.
\end{array}\right.\end{equation}
In words, $u$ is a positive $\widetilde{K}$-harmonic function in $\mathfrak C_\mathfrak L$ which vanishes on 
$\mathfrak L\setminus \mathfrak C_\mathfrak L $ (that is, in particular, on the  boundary of $\mathfrak C_\mathfrak L$. 
\begin{defin}We call harmonic profile for the lattice cone $\mathfrak C_\mathfrak L$ any function $u$  satisfying (\ref{coneprofile}). 
\end{defin}
Such a function is unique up to multiplication by a positive constant. In other words, any two harmonic profiles are constant positive multiples of each other.

The following key property holds. It explains the important role played by the profile: Any other local positive solution vanishing at the boundary can be controlled using the profile.   Consider the following notation.
Let $B=B(\mathbf x,r)$ be the  Euclidean ball of radius $r$ around $\mathbf x$ and set 
$$\mathfrak B_\mathfrak L= B\cap \mathfrak C_\mathfrak L,\;\;
\mathfrak B'_\mathfrak L= B(\mathbf x,r/2)\cap \mathfrak C_\mathfrak L,
$$
$$\overline{\mathfrak B}_\mathfrak L=\{\mathbf x\in \mathfrak L: \mathbf x \mbox{ has a lattice neighbor in } \mathfrak B_\mathfrak L\},$$ and 
$$\partial^*\mathfrak B_\mathfrak L=\{\mathbf x\in \mathfrak L\setminus \mathfrak C_\mathfrak L: \mathbf x \mbox{ has a lattice neighbor in } \mathfrak B\}.$$

In words, $\mathfrak B_\mathfrak L$ is the trace of $B(\mathbf x,r)$ in the lattice cone $\mathfrak C_\mathfrak L$,  $\mathfrak B'_\mathfrak L$ is the trace of $B(\mathbf x,r/2)$ in that  lattice cone,
$\overline{\mathfrak B}_\mathfrak L$ is $\mathfrak B_\mathfrak L$ together with all the lattice points that have a lattice-neighbor in $\mathfrak B_\mathfrak L$, and $\partial^*\mathfrak B_\mathfrak L$ is the part of  $\overline{\mathfrak B}_\mathfrak L\setminus \mathfrak B_\mathfrak L
$ that lies on the boundary of the cone $C$. 

\begin{thm} \label{thH} Fix a constant $C_0$. Let $\mathbf x$ be a point in $\mathfrak C_\mathfrak L$ and $r>0$. 
Let $v$ and $w$ be two positive functions in $\mathfrak B$ which are defined on $\overline{\mathfrak B}_\mathfrak L$ and  vanish along $\partial^*\mathfrak B_\mathfrak L$. Assume that $v$ and $w$ are solutions of
$$(I-\widetilde{K})v= pv,\;\;(I-\widetilde{K})w= qw \mbox{ in } \mathfrak B_\mathfrak L$$
where $p,q$ are functions satisfying $|p|,|q|\le C_0/r^2.$
Then there is a constant $C$ (depending on $k$ and $C_0$ but not on $N$, $\mathbf x,$ $r$, $v$, $w$) such that
$$\sup_{\mathfrak B'_\mathfrak L}\left\{\frac{v}{w}\right\}\le C\inf_{\mathfrak B'_\mathfrak L}\left\{\frac{v}{w}\right\}.$$
\end{thm}
This is a deep and rather intricate theorem and we explain  where it comes from. First, without loss of generality, we can assume that $p=0$ as the general result is a consequence of this special case applied twice. In the context of strictly local Dirichlet spaces, analogous results are given in \cite{LierlSC,LierlH}.  The direct proof in the discrete case goes as follows.  Assume $p=0$. Use the function $v$ to perform a Doob-transform by setting $\widetilde{K}_v(\mathbf x,\mathbf y)=\frac{1}{v(\mathbf x)}\widetilde{K}(\mathbf x,\mathbf y)v(\mathbf y)$, $\mathbf x,\mathbf y\in \mathfrak B$. This is a reversible Markov chain in $\mathfrak B_\mathfrak L$ with reversible measure $v^2$. Moreover, it is a Harnack chain (see \cite{DHESC1,DHESC2}). A version of the elliptic Harnack inequality for the solution $w/v$ of $(I-\widetilde{K})(w/v)=q(w/v)$ gives the desired result. 

\begin{cor} \label{cor-uphi}Let $u$ be a harmonic profile for $\mathfrak C_\mathfrak L$. There exists a constant $C_k\ge 1$ such that, setting $$\mathfrak V_N=\mathfrak S_N\cap \{x_k\ge N/k\},$$  
$$ u(\mathbf x) \frac{\phi_0(\mathbf y)}{C_k u(\mathbf y)}\le \phi_0(\mathbf x)\le u(\mathbf x) 
 \frac{C_k\phi_0(\mathbf y)}{u(\mathbf y)} \mbox{ for all } \mathbf x,\mathbf y\in \mathfrak V_N.$$
\end{cor}
This corollary is a special case of \cite[Theorem 8.13]{DHESC1}. It also follows from Theorem \ref{thH} above.  The function $\phi_0$ has the same behavior in each of the corners of the simplex $\mathfrak S_N$, and, in each of these corners, this behavior is comparable to the behavior of the function $u$ near the tip of the cone $\mathfrak C$.
The region $\mathfrak V_N$ is pictured in Figure \ref{fig4}.

\subsection{From the lattice cone to the continuous cone}
In general, there is no easy way to compute the behavior of the lattice cone profile $u$. However, consider the following continuous version of the profile.

\begin{defin}Let $V_\Omega=\{x: x=r\theta, r>0, \theta\in \Omega \subseteq \mathbb S^{d-1}\}$ be an open half-cone  in Euclidean space $\mathbb R^d$ with base $\Omega$.  Call harmonic profile for (the continuous half-cone) $V_\Omega$, any function $h_V: V\to (0,+\infty)$ which is harmonic in $V$ (that is, is $\mathcal C^2$ in $V$ and satisfies $\Delta h_V=0$ in $V$) and vanishes continuously at the boundary of the cone.
 \end{defin}
Without assumption on the base $\Omega$, a cone may admit many very different profiles. Here we are only interested with the case when $\Omega$ is a nice connected polygonal subset of the sphere $\mathbb S^{d-1}$ in $\mathbb R^d$. 
See, e.g.,  Figure \ref{fig3}.

\begin{prop}[Folklore] \label{prop-alpha}Let $\Omega$ be an open connected subset of $\mathbb S^{d-1}$. Consider the Perron-Frobenius eigenvalue and eigenfunction $\lambda_\Omega,\psi_\Omega$ of the sphere-Laplacian with zero boundary condition on $\Omega$. Then the function
$$x=r\theta\mapsto h_V(x)=r^{\alpha_V}\psi_\Omega(\theta),\;\;\alpha_V=\sqrt{((d/2)-1)^2+\lambda_\Omega}-((d/2)-1) $$
is a harmonic profile for $V$.
\end{prop}

\begin{figure}
    \centering

\begin{tikzpicture}[tdplot_main_coords, scale=4]

\shade[ball color = lightgray,
    opacity = 0.5
] (0,0,0) circle (28.5pt);

\draw[color = gray] (0,0,0) circle (1);

\coordinate (O) at (0,0,0);
\coordinate (A) at (1,0,0);
\coordinate (B) at (0.5, 0.866025, 0);
\coordinate (C) at (0.5, 0.288675, 0.8164966);



\fill[fill=black] (O) circle (0.3pt);
\fill[fill=black] (A) circle (0.3pt);
\fill[fill=black] (B) circle (0.3pt);
\fill[fill=black] (C) circle (0.3pt);


\draw[dashed] (O) -- (B);
\draw (O) -- (C);
\draw (O) -- (A);
\draw (A) -- (B);
\draw (B) -- (C);
\draw (A) -- (C);


\fill[opacity=0.4] (A) -- (B) -- (C) -- cycle;

\fill[opacity=0.4] (O) -- (A) -- (C) -- cycle;

\fill[opacity=0.4] (A) arc (0:60:1) -- cycle;
\tdplotdrawarc[thick]{(O)}{1}{0}{60}{}{}

\tdplotsetrotatedcoords{-90}{70.25}{0}
\tdplotdrawarc[tdplot_rotated_coords, thick]{(O)}{1}{90}{150}{}{}
\fill[tdplot_rotated_coords, color=white, opacity=0.3] (O) -- (A) arc (90:150:1) -- cycle;

\tdplotsetrotatedcoords{-30}{109.7356103}{0}
\fill[tdplot_rotated_coords, opacity=0.4] (B) arc (90:150:1) -- cycle;
\tdplotdrawarc[tdplot_rotated_coords, thick]{(O)}{1}{90}{150}{}{}



\end{tikzpicture}

    \caption{The harmonic profile of a continuous cone is computed using the Perron-Frobenius eigenvalue and eigenfunction of the spherical Laplacian with Dirichlet boundary condition on the spherical base of the cone. In this figure, the cone is associated with the tetrahedron and the base is a particular equilateral triangle on the sphere, the equilateral triangle corresponding to the vertices of the tetrahedron.}
    \label{fig3}
\end{figure}
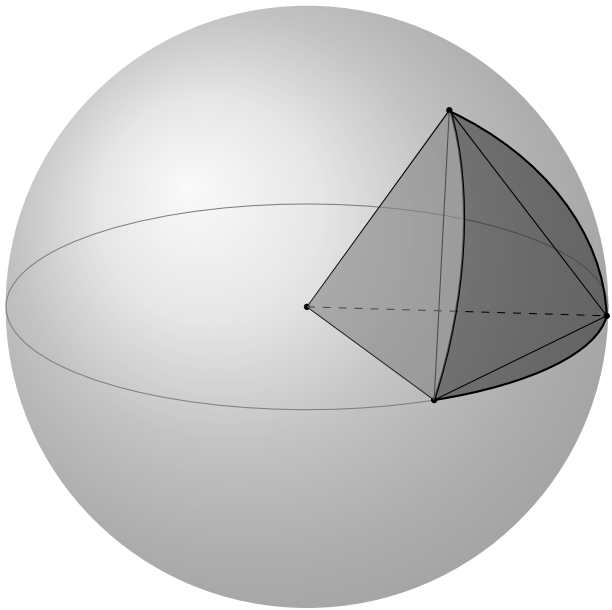

The works \cite{Var,DW1,DW2} study the harmonic profiles of cones (and more general sets in the case of \cite{Var}) and provide very useful comparison between the harmonic profile of a discrete cones and that of the associated continuous cone.  For simplicity, in $d$-dimensional Euclidean space, $\mathbb R^d$ consider a co-compact  lattice  (discrete subgroup) $L$ with the property that simple random walk on $L$ has covariance matrix the identity (or proportional to the identity). 

Let $V_\Omega=\{x: x=r\theta, r>0, \theta\in \Omega \subseteq \mathbb S^{d-1}\}$ be an open half-cone and assume that $\Omega$ is a convex polygonal subset of the sphere. Let
$V_L^*=L\cap V$ be the associated lattice-cone. A (discrete) harmonic profile for that lattice-cone is a function $h_{V,L}$ on the lattice $L$ which vanishes on $L\setminus V^*_L$ and non-negative harmonic in $V^*_L$ and not identically $0$. Here harmonic is with respect to the Markov operator associated with simple random walk on $L$. Amongst the results of \cite{Var,DW2} is the following fact of importance to us. There are positive constants $c,C$ and $D$ such that, for any fix $x_0$
in $V^*_L$ at distance at least $D$ from $\mathbb R^d\setminus V$, we have $h_{V,L}(x_0)>0$ and, for all $x\in V^*_L$ at distance at least $A$ from $\mathbb R^d\setminus V$,
\begin{equation} \label{comp1}c\frac{h_V(x)}{h_V(x_0)}\le \frac{h_{V,L}(x)}{h_{V,L}(x_0)}\le C\frac{h_V(x)}{h_V(x_0)} .
    \end{equation}
The reader should note that, in general, the lattice $L$ and the cone $V$ do not have to be neatly positioned with respect to each other. This means that $V^*_L$, viewed has a subgraph of $L$
can have isolated points near the tip of the cone (at such point, $h_{V,L}$ has to vanish). The estimate (\ref{comp1}) provide a uniform comparison of $h_V$ and $h_{V,L}$ a few steps away from the boundary.
In general, it is indeed possible that there is a sequence of  points $(x_j)_1^\infty$ in $V^*_L$ whose distance to the boundary is positive but tends to $0$ as $j$ tends to infinity. This explain the role of the constant $A$.  

We now return to the particular case of interest to us, that is,  the continuous open half-cone 
$$\mathcal C=\left\{\mathbf x= (t\xi_1,\dots, t\xi_{k-1}, N(1-t)): t> 0, \xi_i>0, 1\le i\le k-1, \sum_1^{k-1} \xi_i=N\right\}$$
and 
the lattice cone
$$\mathfrak C_\mathfrak L=\mathcal C\cap \left\{\mathbf x=(x_1,\dots,x_k): \mathbf x=\sum_{1\le i<j\le k} m_{ij}(e_i-e_j), m_{ij}\in\mathbb Z\right\}.$$ 
Recall that the tip of this cone is the point $\mathbf t_N$ with coordinates $t_1=\dots=t_{k-1}=0,t_k=N$. In this case, the continuous cone and the lattice are neatly arranged in the sense that all the lattice boundary points of 
 $\mathfrak C_\mathfrak L$ lies exactly on the topological boundary of the continuous cone $\mathcal C$ and no points in $\mathfrak C_\mathfrak L$ (viewed as a subgraph of $\mathfrak L$) are isolated. Moreover there exists $\epsilon_k>0$ such that any point in  $\mathfrak C_\mathfrak L$ is at distance at least $\epsilon_k$ from the boundary of $\mathcal C$. This allows us to rephrase (\ref{comp1}) in the following simplified form.
 Set $$\mathbf t'_N=(1,\dots,1,N-k+1)\in \mathfrak C_\mathfrak L.$$
This point stands in the lattice cone, closest to the tip $\mathbf t_N$. 
 \begin{prop} \label{prop-hh}Let $h$ and $h_\mathfrak L$ be, respectively the (continuous) harmonic profile of the continuous cone $\mathcal C$ and the (lattice) harmonic profile of the 
 lattice-cone $\mathfrak C_\mathfrak L$, both normalized by the condition
 $$h(\mathbf t'_N)=h_{\mathfrak L}(\mathbf t'_N)=1.$$ There are constants $c,C\in (0,+\infty)$ such that, for all $\mathbf x\in \mathfrak C_\mathfrak L$,
 $$c h(\mathbf x)\le h_\mathfrak L(\mathbf x)\le Ch(\mathbf x).$$
 \end{prop}
 
\subsection{Estimating $\phi_0$ in terms of the profile $h$ of the continuous cone}

Let $\lambda_k>0$ be the Perron-Frobenius eigenvalue (i.e., lowest eigenvalue) of the (positive) sphere-Laplacian with Dirichlet boundary condition in the regular spherical simplex $\Omega_k$ of dimension $k-2$
cut by the cone $\mathcal C$ on the unit sphere $\mathbb S^{k-2}$ of the $k-1$-space $\{\mathbf x: \sum_1^kx_i=N\}$ viewed has a vector space with origin at $\mathbf t_N=(0,\dots,0,N)$. This number is defined by the variational formula
\begin{equation}\label{lambda}\lambda_k=\inf\left\{\frac{\int_{\Omega_k}|\nabla f|^2 d\sigma}{\int_{\Omega_k}|f|^2d\sigma}: f\in \mathcal C^\infty_c(\Omega_k)\right\}.\end{equation}
\begin{defin}
\label{def-alpha} Set  $$\alpha_{k}=\sqrt{(((k-1)/2)-1)^2+\lambda_k}-(((k-1)/2)-1).$$
\end{defin}

Recall that $\phi_0$ is the (positive, normalized) Perron-Frobenious eigenfunction of $K$ on the lattice simplex $\mathfrak S_N$. See (\ref{defK}-(\ref{defphi0}).  Because of the uniqueness of $\phi_0$ and  the obvious symmetry of the problem
under permutation of the coordinates of a point $\mathbf x=(x_1,\dots,x_k)$, the function $\phi_0$
is symmetric under any permutation of these coordinates. Because $\sum _1^kx_i=N$ for any point in  $\mathfrak S_N$, at least one of the $x_i$'s, $1\le i\le k$, is larger or equal to $N/k$ and we can assume without loss of generality that $x_k\ge N/k$. See Figure \ref{fig4}.

\begin{figure}
    \centering
\begin{tikzpicture}[tdplot_main_coords, scale=2]


\coordinate (O) at (0,0,0);
\coordinate (A) at (0.25,0,0);
\coordinate (B) at (0.25*0.5, 0.25*0.866025, 0);
\coordinate (C) at (0.25*0.5, 0.25*0.288675, 0.25*0.8164966);

\coordinate (A2) at (4, 0,0);
\coordinate (B2) at (2, 4*0.866025 , 0);
\coordinate (C2) at (2, 4*0.288675 , 4* 0.8164966 );



\draw[dashed] (O) -- (2.5, 5*0.866025 , 0);
\draw (O) -- (2.5, 5*0.288675 , 5* 0.8164966 );
\draw (O) -- (5, 0,0);

\fill[opacity=0.2] (O) -- (A2) -- (B2) -- (C2) -- cycle;


\foreach \x in {0.1, 0.2, 0.3, 0.4, 0.5, 0.6, 0.7, 0.8, 0.9, 1}
\foreach \y in {4}
{\draw[opacity=0.5] (\y - \y*\x + \y*\x*0.5,  \y*\x*0.866025, 0) 
-- (\y - \y*\x  + \y*\x*0.5,  \y*\x*0.288675, \y*\x*0.8164966 ); 
\draw[opacity=0.5] (\y*0.5 - \y*\x*0.5 + \y*\x ,  \y*0.866025 - \y*\x*0.866025, 0  ) 
-- (\y*0.5 - \y*\x*0.5 + \y*\x*0.5,  \y*0.866025 - \y*\x*0.866025 + \y*\x*0.288675 , \y*\x* 0.8164966  ); 
\draw[opacity=0.5] (\y*0.5 - \y*\x*0.5 + \y*\x ,  \y*0.288675 - \y*\x*0.288675 , \y*0.8164966 - \y*\x*0.8164966 ) 
-- (\y*0.5 - \y*\x*0.5 + \y*\x*0.5 ,  \y*0.288675 - \y*\x*0.288675 + \y*\x*0.866025 , \y* 0.8164966 - \y*\x* 0.8164966 ); 
}

\fill[opacity=0.2] (O) -- (3*1, 0,0) -- (3*0.5, 3*0.866025 , 0) -- (3*0.5, 3*0.288675 , 3* 0.8164966 ) -- cycle;
\draw (O) -- (3*0.5, 3*0.288675 , 3* 0.8164966 );
\draw (3*1, 0,0) -- (3*0.5, 3*0.866025 , 0);
\draw (3*0.5, 3*0.866025 , 0) -- (3*0.5, 3*0.288675 , 3* 0.8164966 );
\draw (3*1, 0,0) -- (3*0.5, 3*0.288675 , 3* 0.8164966 );

\coordinate (L) at (0.6, 0.288675 , 0.8164966 );

\node at (L) [ left = 0.4cm of L] {$\left\{\mathbf{x} \; : \; x_4 \ge N/4 \right\}$};

\end{tikzpicture}
    \caption{The dark grey area represents $\mathfrak S_n\cap\{\mathbf x: x_k\ge N/k\}$. Here $k=4$ (4-player game) and $N=10$. Note that the union of the zones corresponding to each corner cover the entire tetrahedron.}
    \label{fig4}
\end{figure}
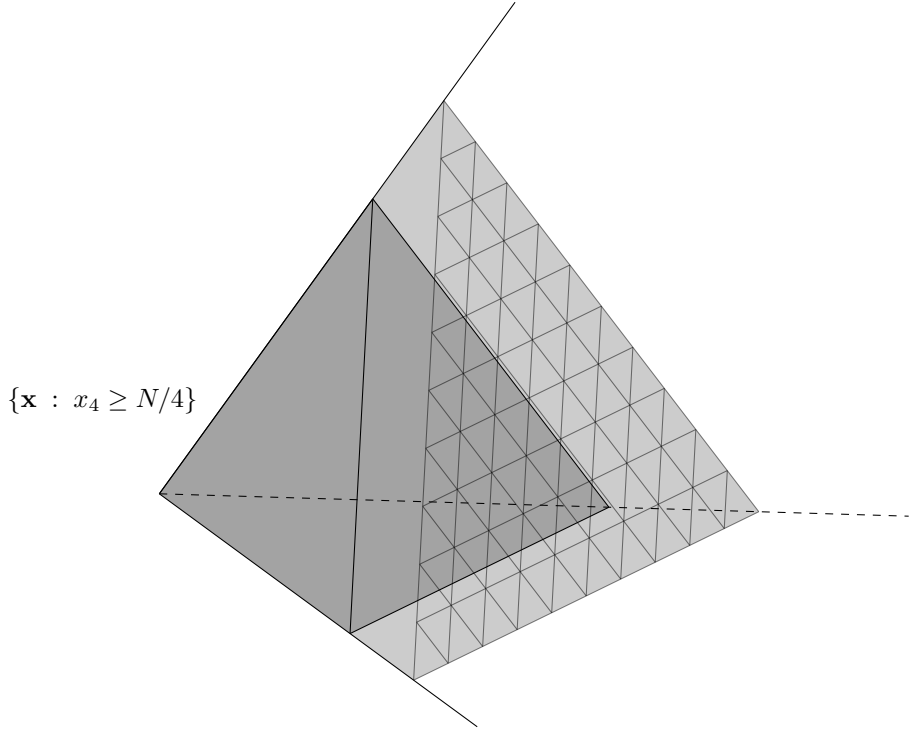

\begin{prop} \label{prop-hphi}There are constants $c_k\le C_k\in (0,+\infty)$ such that, in $ \mathfrak S_N\cap\{\mathbf x: x_k\ge N/k\}$,
$$c_k h(\mathbf x)\le N^{((k-1)/2)+\alpha_k}\phi_0(\mathbf x)\le C_k h(\mathbf x)  $$
\end{prop}
\begin{proof} By Proposition \ref{prop-hh} and Corollary \ref{cor-uphi} with $u=h_\mathfrak L$, we have
$$\phi_0(\mathbf x)\asymp_k \phi_0(\mathbf t'_N) h(\mathbf x).$$
According to \cite[(5.16)]{DHESC2}, we also have $\phi_0(\mathbf o_N)^2\asymp_k N^{-(k-1)}$ where $\mathbf o_N$ is a lattice point 
nearest to the center $(N/k,\dots,N/k)$ of the simplex $\{\mathbf x=(x_1,\dots,x_k): x_i>0, \sum_1^k x_i=N\}$.   This gives
$$N^{-(k-1)/2}\asymp_k\phi_0(\mathbf o_N)\asymp_k \phi_0(\mathbf t'_N) h(\mathbf o_N)\asymp_k  \phi_0(\mathbf t'_N) N^{\alpha_k},$$
so that, as desired,  
$$\phi_0(\mathbf x) \asymp_k N^{-(\alpha_k+((k-1)/2))} h(\mathbf x).$$
\end{proof}

\section{Harmonic profile and $\phi_0$ in coordinates} 
\subsection{The continuous harmonic profile $h$ in coordinates when $k=4$}
The goal of this section is to provide (explicit) estimates for the continuous harmonic profile $h$ of the continuous cone $$\mathcal C=\left\{\mathbf x= (t \xi_1,t \xi_2, t \xi_3, N(1-t)): t> 0, \xi_i>0, 1\le i\le 3, \xi_1+\xi_2+\xi_3=N\right\}$$
in terms of the (free) coordinates $(x_1,\dots,x_{3})$ of the point $\mathbf x=(x_1,x_2,x_3,x_4)\in \mathcal C$.

By proposition \ref{prop-alpha}, we know that $h$ has the form
\begin{equation} \label{h1}
h(\mathbf x)= r^{\alpha_4}\psi_4(\theta)
\end{equation}
where $(r,\theta)$ are the polar coordinate of a point in the hyperplane $\sum_1^4x_i=N$ wit respect to the origin $\mathbf t_4$ (the tip or our cone $\mathcal C$). We postpone the discussion of the value of the real $\alpha_4$ but note that $\alpha_4\approx 5.68$. We note that
the radius $r$ satisfies
\begin{equation}
r\asymp x_1+x_2+x_3 \end{equation}
for any point $\mathbf x=(x_1,x_2,x_3,x_4)$ contained in $\mathcal C$.  So our goal is to understand the the function $\psi_4$. By definition, $\psi$ is the Perron-Frobenius eigenfunction of the spherical Laplacian in the spherical domain $\Omega_4$ cut by our cone $\mathcal C$ on the $2$-dimensional unit sphere centered at $\mathbf t_4$ in  
the $3$-space $\{\mathbf x: \sum_1^4x_i=N\}$. The final result we want to prove reads as follows.
\begin{prop} \label{prop-h}The harmonic profile $h$ of the cone $\mathcal C$
satisfies
$$ h(\mathbf x)\asymp (x_1+x_2+x_3)^{\alpha-3\beta +3}
  [(x_1+x_2) (x_1+x_3)(x_2+x_3)]^{\beta-2}x_1x_2x_3.
$$
where $\alpha=\alpha_4\approx 5.68$ is as defined above and $\beta=\pi/\arccos (1/3)\approx 2.55$.
\end{prop}
Before embarking with the proof, we make the following observations. Working in  Euclidean space $\mathbb R^m$ with canonical coordinates $(y_1,\dots,y_m)$, few cones have harmonic profile whose expression in coordinates are simple and explicit. 

In dimension $2$, $\{\mathbf y=(y_1,y_2): y_1\in \mathbb R, y_2>0\}$ has $h_\pi(\mathbf y)=y_2$ while the cone
$\{\mathbf y=(y_1,y_2): y_1>0, y_2>0\}$ has $h_{\pi/2}(\mathbf y)=y_1y_2$. For any $\eta\in (0,2\pi)$, the cone $$V_\eta=\{\mathbf y=re^{i\theta}: r>0, 0<\theta<\eta\}$$ of aperture $\eta$  has 
$$h_\eta(\mathbf y)=r^{\pi/\eta}\sin (\pi \theta/\eta) .$$
Even so there is no easy formula for $h$ in terms of $(y_1,y_2)$, it is helpful to observe that
$$h(\mathbf y)\asymp (y_1^2+y_2^2)^{(\pi/2\eta)-1 } d(\mathbf y,L_0)d(\mathbf y, L_{\eta})$$
where $L_\theta$ is the half-line $\{\mathbf z=re^{i\theta}: r>0\}$ and $d(\mathbf y,L)$ is the distance from $\mathbf y$ to $L$. In such an estimate, the terms of the form $d(\mathbf y,L)$ can be replaced by any equivalent quantity. 
For instance, for $\eta\in (0,3\pi/4]$, we can write
$$h_\eta(\mathbf y)\asymp (y_1+y_2)^{(\pi/\eta)-2 } (y_1\sin\theta-y_2\cos\theta)y_2.$$
In the $2$-dimensional cone 
$$\mathcal C=\left\{\mathbf x= (t \xi_1,t \xi_2, N(1-t)): t> 0, \xi_1,\xi_2>0, \xi_1+\xi_2=N\right\}$$
which has aperture $\pi/3$, with $\mathbf x=(x_1,x_2,x_3)$, $x_1+x_2+x_3=N$, this give
$$h_{\pi/3}(\mathbf x)\asymp (x_1+x_2)x_1x_2$$
where $\asymp$ is uniform in $N$. See Figure \ref{fig6}.

\begin{figure}
    \centering
\begin{tikzpicture}[tdplot_main_coords, scale=4]

\coordinate (O) at (0,0,0);
\coordinate (X) at (1.2,0,0);
\coordinate (Y) at (0,1.2,0);
\coordinate (Z) at (0,0,1.2);

\coordinate (extY) at (0, 2, -0.8);
\coordinate (extZ) at (2, 0, -0.8);

\draw[thick,->] (0,0,0) -- (3.8,0,0) node[anchor=north east]{$x_1$};
\draw[thick,->] (0,0,0) -- (0,2.4,0) node[anchor=north west]{$x_2$};
\draw[thick,->] (0,0,0) -- (0,0,1.5) node[anchor=south]{$x_3$};

\draw (X) -- (Y);
\draw (Y) -- (Z);
\draw (Z) -- (X);

\draw (Y) -- (extY);
\draw (Z) -- (extZ);

\fill[tdplot_main_coords, opacity=0.2] (X) -- (Y) -- (Z) -- cycle;

\fill[tdplot_main_coords, opacity=0.1] (X) -- (Y) -- (extY) -- (extZ) -- cycle;

{
\draw[opacity=0.6] (0,1.8,-0.6) -- (0.2, 1.8, -0.8);
\draw[opacity=0.6] (0, 1.6,-0.4) -- (0.4, 1.6, -0.8);
\draw[opacity=0.6] (0, 1.4,-0.2) -- (0.6, 1.4, -0.8);
\draw[opacity=0.6] (0,1.2,0) -- (0.8, 1.2, -0.8);
\draw[opacity=0.6] (0,1,0.2) -- (1, 1, -0.8);
\draw[opacity=0.6] (0,0.8,0.4) -- (1.2, 0.8, -0.8);
\draw[opacity=0.6] (0,0.6,0.6) -- (1.4, 0.6, -0.8);
\draw[opacity=0.6] (0,0.4,0.8) -- (1.6, 0.4, -0.8);
\draw[opacity=0.6] (0,0.2,1) -- (1.8, 0.2, -0.8);

\draw[opacity=0.6] (0,1,0.2) -- (1,0,0.2);
\draw[opacity=0.6] (0,0.8,0.4) -- (0.8,0,0.4);
\draw[opacity=0.6] (0,0.6,0.6) -- (0.6,0,0.6);
\draw[opacity=0.6] (0,0.4,0.8) -- (0.4,0,0.8);
\draw[opacity=0.6] (0,0.2,1) -- (0.2,0,1);
\draw[opacity=0.6] (0, 1.4 ,-0.2) -- (1.4, 0, -0.2);
\draw[opacity=0.6] (0, 1.6, -0.4) -- (1.6, 0, -0.4);
\draw[opacity=0.6] (0, 1.8, -0.6) -- (1.8, 0,-0.6);

\draw[opacity=0.6] (1.8, 0.2, -0.8) -- (1.8, 0, -0.6);
\draw[opacity=0.6] (1.6, 0.4, -0.8) -- (1.6, 0, -0.4);
\draw[opacity=0.6] (1.4, 0.6, -0.8) -- (1.4,0 ,-0.2);
\draw[opacity=0.6] (1.2, 0.8, -0.8) -- (1.2,0, 0);
\draw[opacity=0.6] (1,1,-0.8) -- (1,0,0.2);
\draw[opacity=0.6] (0.8,1.2,-0.8) -- (0.8,0,0.4);
\draw[opacity=0.6] (0.6, 1.4,-0.8) -- (0.6,0,0.6);
\draw[opacity=0.6] (0.4,1.6,-0.8) -- (0.4,0,0.8);
\draw[opacity=0.6] (0.2,1.8,-0.8) -- (0.2,0,1);
\draw[opacity=0.6] (0, 2, -0.8) -- (0, 0, 1.2);

}

\end{tikzpicture}

    \caption{The triangle (dark grey), cone and lattice for the 3-player game: the harmonic profile satisfies $h(\mathbf x)\asymp (x_1+x_2)x_1x_2$, $\mathbf x=(x_1,x_2,x_3)$, $\sum_1^3 x_i=N$. Here $N=6$.}
    \label{fig6}
\end{figure}
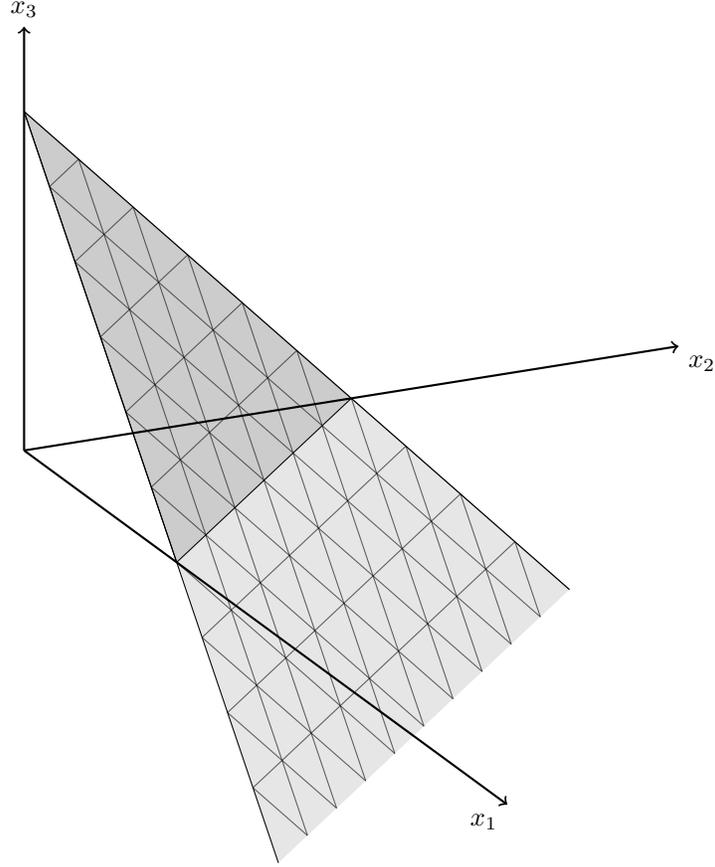

Such two dimensional computations apply to a wedge determined by two half plane meeting along their edge (a copy of $\mathbb R$) in three dimension
 after choosing coordinate wisely. Namely, one can represent such a wedge as  $W=\mathbb R\times V_\eta$ where $V_\eta$ lies in a plane orthogonal to the mentioned copy of $\mathbb R$.
 If we call $(z,y_1,y_2)$ the coordinate of $\mathbb R^3$
 corresponding to $\mathbb R\times V_\eta$, the profile of this wedge is  $h_W(z,y_1,y_2)=r^{\pi/\eta}\sin (\pi \theta/\eta)$ where $r,\theta$ are the polar coordinates of the point $(y_1,y_2)$. Hence, we have
 $$h_W(z,y_1,y_2)\asymp (y_1+y_2)^{(\pi/\eta)-2}(y_1\sin \theta-y_2\cos\theta)y_2.$$
 
 \begin{proof}[Proof of Proposition \ref{prop-h}]
 We now explain how to use this information to express $h$ in (\ref{h1}) purely in terms of $x_1,x_2,x_3$.  Consider a region in the cone $\mathcal C$ centered around one of the tips of the spherical triangle $\mathcal C\cap \mathbb S^2(\mathbf t)$, $\mathbf t=\mathbf t_4=(0,0,0,N)$, say
 $\mathbf q_1=(1,0,0,N-1))$.  In an Euclidean ball $A$ of radius $1/2$ around this point,
  $h$ is a positive harmonic function in a wedge $W$ with aperture equal to the angle between the normal vectors of the two planes $Q_2, Q_3$ meeting along the edge equal to the line $L_1$ passing through $\mathbf t$ and $\mathbf q_1$. The plane $Q_i$ is the plane determined by the three points $\mathbf t, \mathbf q_1, \mathbf q_i$,  parallel to the plane spanned by the vectors $\overrightarrow{\mathbf t\mathbf q_1},\overrightarrow{\mathbf t\mathbf q_i}$, with $\overrightarrow{\mathbf t\mathbf q_1}=(1,0,0,-1)$, $\overrightarrow{\mathbf t\mathbf q_2}=(0,1,0,-1)$, $\overrightarrow{\mathbf t\mathbf q_3}=(0,0,1,-1)$. It follows that the vectors
  $(1,1,-3,1)$ and $(1,-3,1,1)$ are, respectively, normal to $Q_2$ and $Q_3$ and contained in $\{\sum _1^4x_i=0\}$. The cosine of their angle $\eta\in [0,\pi)$ is
  $\cos \eta= 4/12=1/3$. Set 
  $$\beta=\frac{\pi}{\arccos (1/3)}\approx \frac{\pi}{1.23}\approx 2.55.$$
  For a point $\boldsymbol{\zeta}=(\zeta_1,\zeta_2,\zeta_3,\zeta_4)\in \mathcal C \cap A$, $\zeta_1\asymp 1,\zeta_4\asymp N$, $\zeta_2,\zeta_3\in (0,1/2)$,  and
  $$h_W(\boldsymbol \zeta) \asymp (\zeta_2+\zeta_3)^{\beta-2}\zeta_2\zeta_3.$$
  This gives that, in the ball $\frac{1}{2}A$,
  $$\frac{h(\mathbf x)}{h(\mathbf y)}\le C \frac{h_W(\mathbf x)}{h_W(\mathbf y)}.$$
  Picking $\mathbf y \asymp (1/\sqrt{3},1/\sqrt{3},1/\sqrt{3},N-3/\sqrt{3})$ gives
 $$h(\mathbf x) \le C \psi (\theta)
 (x_2+x_3)^{\beta-2}x_2x_3.$$ 
 exchanging the role of $\mathbf x$ and $\mathbf y$ gives a matching lower bound so that
 $$h(\mathbf x) \asymp C \psi (\theta)
 (x_2+x_3)^{\beta-2}x_2x_3.$$ 
 
  By symmetry, this gives, on $\mathcal C\cap \mathbb S^2(\mathbf t)$, $\mathbf x=r\theta$, $r=1$,
  $$\psi(\theta)\asymp (x_1+x_2)^{\beta-2} (x_1+x_3)^{\beta-2}(x_2+x_3)^{\beta-2}x_1x_2x_3.$$
 Finally, for any $\mathbf x\in \mathcal C$,
$$h(\mathbf x)\asymp (x_1+x_2+x_3)^{\alpha-3\beta +3}
  (x_1+x_2)^{\beta-2} (x_1+x_3)^{\beta-2}(x_2+x_3)^{\beta-2}x_1x_2x_3.$$
  \end{proof}
  
  \subsection{The function $\phi_0$ in coordinates}
With Proposition \ref{prop-h} at hands, it is an easy matter to estimate the Perron-Frobenius function $\phi_0$ of the simplex $\mathfrak S_n$ in the four-player case ($k=4$).
For this purpose we define the symmetric functions in the variables $x_1,x_2,x_3,x_4$:
\begin{eqnarray*}
\tau_1&=&x_1x_2x_3x_4,\\ \tau_2&=&(x_1+x_2)(x_1+x_3)(x_1+x_4)(x_2+x_3)(x_2+x_4)(x_3+x_4),\\
\tau_3&=&(x_1+x_2+x_3)(x_1+x_2+x_4)(x_1+x_3+x_4)(x_2+x_3+x_4).
\end{eqnarray*}

\begin{thm}\label{thm-tau} In the $4$-player case, the Perron-Frobenius function $\phi_0$ of the simplex $\mathfrak S_N$ satisfies
\begin{equation}\label{tau}\forall\,\mathbf x\in \mathfrak S_N,\;\;\;\phi_0(\mathbf x)\asymp N^{-((3/2)+4\alpha-6\beta+4)}
\tau_3^{\alpha-3\beta+3}\tau_2^{\beta-2}\tau_1 \end{equation}
uniformly in $N$.  If $\mathbf x=(x_1,\dots,x_4$ satisfies $x_4\ge N/4 $, this simplifies to
$$\phi_0(\mathbf x)\asymp N^{-((3/2)+\alpha)}(x_1+x_2+x_3)^{\alpha-3\beta +3}
  (x_1+x_2)^{\beta-2} (x_1+x_3)^{\beta-2}(x_2+x_3)^{\beta-2}x_1x_2x_3.
$$
\end{thm}
        \begin{proof} We simply need to check that (\ref{tau}) is compatible with Propositions \ref{prop-h} and  \ref{prop-hphi}. Let start with the central part of $\mathfrak S_N$ where each 
        $x_i$ satisfies $x_i\asymp N$.  By the basic Harnack inequality and the normalization $\|\phi_0\|_2^2=1$, we know that $\phi_0(\mathbf x)\asymp N^{-3/2}$ there.  In that region, the expression
$$T=\tau_3^{\alpha-3\beta+3}\tau_2^{\beta-2}\tau_1$$ is of order $$N^{4(\alpha-3\beta+3)+6(\beta-2)+4}=N^{4\alpha-6\beta+4}.$$
This shows that  (\ref{tau}) amounts to $\phi_0(\mathbf x)\asymp N^{-3/2}$ in this middle part, as desired.  Then we focus  on each of the four corners of $\mathfrak S_N$. By symmetry, it suffices to consider one of this corner, say, the corner where $x_4\ge N/4$. In that corner, the expression
$T$
satisfies
$$T\asymp N^{3\alpha-6\beta+4} (x_1+x_2+x_3)^{\alpha-3\beta +3}
  (x_1+x_2)^{\beta-2} (x_1+x_3)^{\beta-2}(x_2+x_3)^{\beta-2}x_1x_2x_3.$$
Hence, the right-hand side of (\ref{tau}) becomes
$$N^{-((3/2)+\alpha) }(x_1+x_2+x_3)^{\alpha-3\beta +3}
  (x_1+x_2)^{\beta-2} (x_1+x_3)^{\beta-2}(x_2+x_3)^{\beta-2}x_1x_2x_3.$$
This is indeed what Propositions \ref{prop-h} and  \ref{prop-hphi} entail. It also gives the announced approximation for $\phi_0$ in the corner $x_4\ge N/4$.
\end{proof}
\subsection{The approximate computation of $\alpha$} 
As explained above (Definition \ref{def-alpha}), computing $\alpha=\alpha_4$ is equivalent to computing
the Dirichlet eigenvalue $\lambda_4$ at (\ref{lambda}) for the equilateral spherical triangle obtained on the unit sphere $\mathbb S^2$ by drawing a unit simplex in $\mathbb R^3$ with one vertex at the origin of $\mathbb R^3$. The exact value for this eigenvalue is not known. Grady Wright has constructed a numerical algorithm that approximates $\lambda$ by computing the eigenvalue of a finite matrix corresponding to a radial basis function (RBF), finite difference approximation using a carefully selected grid \cite{SWKF,GW}. One of the difficulties is associated with dealing with corners of the spherical triangle, and the grid is selected to be more clustered near these corners.

When discretizing the continuous problem to the matrix eigenvalue problem, the "symmetry" of the original problem is lost and one is led to the computation of the spectrum (for us, just the lowest eigenvalue) of a non-symmetric matrix. In general, such computations are known to be difficult as the spectrum of the non-symmetric matrix can be more sensitive to perturbations than the underlying continuous problem. This is one of the reasons for which, although one should be confident that the value given by the algorithm is a good approximation, there is no proof of it and no error estimate. One way to ``check'' the algorithm is by testing its result on the equilateral spherical triangle associated with the first quadrant, which has eigenvalue 12. For the triangle associated with the tetrahedron, The algorithm gives 11.99. Grady Wright's algorithm gives $\lambda_4\approx 38.447$ and produces Figure \ref{fig5}.

\begin{figure}[h]
  \includegraphics[width=0.5\linewidth]{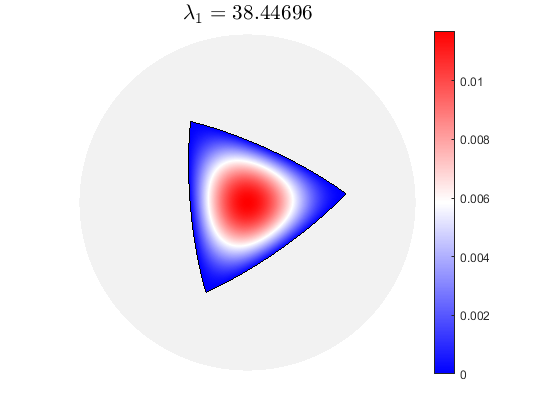}
  \caption{The eigenvalue $\lambda_4$ (for the 4-player game; called $\lambda_1$ in the figure) and the associated eigenfunction}
  \label{fig5}
\end{figure}

\section{Applications} \label{sec-Appli}
In this section, we describe various explicit estimates obtained by applying the knowledge of $\phi_0$. Recall that $\alpha\approx 5.68$ and $\beta=(\pi/\arccos{1/3})\approx 2.55$.
\subsection{General estimate}
Putting together Theorem \ref{th2} and formula (\ref{tau}) yields the following general estimate.
In order to simplify formula, we use the natural symmetries to focus on three cases depending on the positions of the starting point $\mathbf s\in \mathfrak S_N$ and the target point $\mathbf z\in \partial \mathfrak S_N$.
Define three cases as follows:
\begin{itemize}
    \item [Case 1] Assume $s_4\ge N/4$, $z_4=0$, and $z_3\ge N/3$  (one of the above average players ends up loosing). In this case the distance between $\mathbf s$ and $\mathbf z$ is of order $N$.
    \item [Case 2] Assume $s_4\ge N/4>\max\{s_1,s_2,s_3\}$, and $z_3=0$, $z_2\ge N/3$, and $d(\mathbf s,\mathbf z)\asymp N$. 
    \item [Case 3] Assume $s_4\ge N/4>\max\{s_1,s_2,s_3\}$, and $z_3=0$, $z_4\ge N/3$ (one of the below average players ends up loosing while one of the above average players remains above average). The distance between $\mathbf s$ and $\mathbf z$ may be small.
\end{itemize}
Note that there is some overlap between cases 2 and 3. 

\begin{thm}\label{thm-first}In the $4$-player game,  
consider points $\mathbf s\in \mathfrak S_N$ and $\mathbf z\in \partial \mathfrak S_N$ with $s_4\ge N/4$. 
Assume Case 1 or Case 2 above. In Case 1, set $z_*=z_2$ and, in Case 2, set $z_*=z_4$. Then we have
\begin{eqnarray*}
\lefteqn{ \mathbf P(\mathbf  X_\tau =\mathbf z  | \mathbf X_0=\mathbf s)\asymp }&&\\
&&N^{-1-2\alpha}(s_1+s_2+s_3)^{\alpha-3\beta+3}[(s_1+s_2)(s_1+s_3)(s_2+s_3)]^{\beta-2}s_1s_2s_3\\
&& \times (z_1+z_*)^{\alpha-2\beta+1}(z_1z_*)^{\beta-1}.\end{eqnarray*}
In particular, 
\begin{eqnarray*}
\lefteqn{\mathbf P(\mathbf  X_\tau \in \mathfrak T_{N,4} | \mathbf X_0=\mathbf s)}&&\\
&&\asymp N^{-\alpha}
(s_1+s_2+s_3)^{\alpha-3\beta+3}[(s_1+s_2)(s_1+s_3)(s_2+s_3)]^{\beta-2}s_1s_2s_3,\end{eqnarray*}
and even
\begin{eqnarray*}
\lefteqn{\mathbf P(\mathbf  X_\tau \in\{\mathbf z: z_4\le N/5\} | \mathbf X_0=\mathbf s)}&&\\
&&\asymp N^{-\alpha}
(s_1+s_2+s_3)^{\alpha-3\beta+3}[(s_1+s_2)(s_1+s_3)(s_2+s_3)]^{\beta-2}s_1s_2s_3.\end{eqnarray*}
\end{thm}
\begin{proof} In cases 1 and $2$ where we are sure that $d(\mathbf s,\mathbf z)\asymp N$, Theorem \ref{th2} and (\ref{P1}) give
$$ \mathbf P(\mathbf  X_\tau =\mathbf z  | \mathbf X_0=\mathbf s)\asymp N^2
\phi_0(\mathbf s)\phi_0(\mathbf y_\mathbf z).
$$
where $\mathbf y_\mathbf z$ is an interior neighbor of the boundary point $\mathbf z$. 
Theorem \ref{thm-tau} gives
$$\phi_0(\mathbf s)\asymp N^{-((3/2)+\alpha)}(s_1+s_2+s_3)^{\alpha-3\beta+3}[(s_1+s_2)(s_1+s_3)(s_2+s_3)]^{\beta-2}s_1s_2s_3.$$

Similarly, the point $\mathbf y_z$ has third coordinate greater than $(N/3)-1$ in case 1 and second coordinate greater than $(N/3)-1$ in case 2 which yields 
$$\phi_0(\mathbf z)\asymp N^{-(3/2)+\alpha} (z_1+z_*)^{\alpha -2\beta+1}(z_1z_*)^{\beta-1}  .$$
This gives the desired result for $ \mathbf P(\mathbf  X_\tau =\mathbf z  | \mathbf X_0=\mathbf s)$. Summing over the free variables $z_1,z_*$ gives the other statements.

\end{proof}

\begin{exa}[Players distributed on a power scale] As an illustration, assume that
$$s_4\asymp N, \mbox{ and } \;s_i\asymp N^{\epsilon_i}, i\in \{1,2,3\}
\mbox{ with } 0\le \epsilon_1\le \epsilon_2\le \epsilon_3<1.$$
What is the (order of magnitude of) probability that the fourth player, which is currently the dominant player because $\epsilon_i\in [0,1]$, $i\in \{1,2,3\}$, ends up loosing first?
The answer is
$$\mathbf P(\mathbf  X_\tau \in \mathfrak T_{N,4} | \mathbf X_0=\mathbf s)
\asymp N^{-\alpha(1-\epsilon_3)-\beta (\epsilon_3-\epsilon_2) -(\epsilon_2-\epsilon_1)}.$$
When $\epsilon_1=\epsilon_2=\epsilon_3$, 
$$\mathbf P(\mathbf  X_\tau \in \mathfrak T_{N,4} | \mathbf X_0=\mathbf s)
\asymp N^{-\alpha(1-\epsilon_3)}.$$ When $\epsilon_2=\epsilon_1=\epsilon_3/2$,
$$\mathbf P(\mathbf  X_\tau \in \mathfrak T_{N,4} | \mathbf X_0=\mathbf s)
\asymp N^{-\alpha (1-\epsilon_3)-\epsilon_3\beta /2 }.$$
\end{exa}

\begin{thm} \label{thm-second}In the $4$-player game,
consider points $\mathbf s\in \mathfrak S_N$ and $\mathbf z\in \partial \mathfrak S_N$ with $s_4\ge N/4. $
Assume Case 3 above with $d=d(\mathbf s,\mathbf z)$, that is $z_3=0,z_4\ge N/3$. Then we have
\begin{eqnarray*}
\lefteqn{\mathbf P(\mathbf  X_\tau =\mathbf z  | \mathbf X_0=\mathbf s) \asymp \frac{(s_1+s_2+s_3)^{\alpha-3\beta+3}}{d(s_1+s_2+s_3+d)^{2(\alpha-3\beta+3)}}}\hspace{1in}&&\\
&&\times \frac{[(s_1+s_2)(s_1+s_3)(s_2+s_3)]^{\beta-2}}{[(s_1+s_2+d)(s_1+s_3+d)(s_2+s_3+d)]^{2(\beta-2)}}\\
&&\times \frac{s_1s_2s_3}{[(s_1+d)(s_2+d)(s_3+d)]^2}\\
&&\times   (z_1+z_2)^{\alpha-2\beta+1}(z_1z_2)^{\beta-1}.\end{eqnarray*}
\end{thm}
\begin{exa}[Most likely outcome in the very dominant player case] 
To illustrate this result, assume $\mathbf s=(1,1,1,N-3)$ and $\mathbf z=(z_1,z_2,0,z_4)$ with $z_4\ge N/3$. In this situation, the result simplify to
\begin{eqnarray*}\mathbf P(\mathbf  X_\tau =\mathbf z  | \mathbf X_0=\mathbf s) &\asymp&
\frac{(z_1+z_2)^{\alpha-2\beta+1}(z_1z_2)^{\beta-1}}{d^{1+2\alpha}}\\
&\asymp & \frac{1}{d^{1+\alpha}}\left(\frac{z_1+z_2}{d}\right)^\alpha\left(\frac{z_1z_2}{(z_1+z_2)^2}\right)^\beta \left(\frac{z_1+z_2}{z_1z_2}\right)
\end{eqnarray*}
In the last expression, every factor is bounded above and this show that most of the mass is obtained when $d\asymp 1$ which forces $z_1\asymp z_2\asymp 1$
 and $|z_4-N|\asymp 1.$
\end{exa}

\begin{exa}[Probability that the second dominant player loose first] We saw that the probability that a very dominant player (say, $s_4=N-3$) looses first is of order $N^{-\alpha}$. If there is a dominant player, i.e., $s_4\sim N$, and a subdominant player, i.e., $s_3\asymp N^\epsilon$, $\epsilon\in (0,1)$, while $s_1\asymp s_2\asymp 1$, the probability that the dominant player looses first is of order $$N^{-\alpha (1-\epsilon)-\beta\epsilon}=N^{-\alpha+(\alpha-\beta)\epsilon}.$$
Theorems \ref{thm-first} and  \ref{thm-second} allows us to estimate the probability that the subdominant player ends up loosing first in this situation. For this, we assume that $s_3\asymp N^{\epsilon}$
with $\epsilon\in (0,1)$, and $s_1,s_2\asymp 1$. Of course, this implies that $s_4\sim N$.  We want to compute the probability that $\mathbf X_\tau\in \mathfrak T_{N,3}=\{\mathbf z\in \partial \mathfrak S_N \cap \{z_3=0\}\}$. 

Given $\mathbf z\in \partial \mathfrak S_N$ with $z_3=0$ and $z_1+z_2\ge 2N/3$, hence $d(\mathbf s,\mathbf z)\asymp N$ and $z_4\le N/3$, we use Theorem \ref{thm-first} to see that
$$\mathbf P(\mathbf  X_\tau =\mathbf z  | \mathbf X_0=\mathbf s) \asymp
N^{-\alpha+\epsilon(\alpha-\beta)-\beta-1}\min\{z_1,z_2\}^{\beta-1}.$$
The contribution of this to $\mathbf P(\mathbf  X_\tau \in \mathfrak T_{N,3}  | \mathbf X_0=\mathbf s)$ is of order $N^{-\alpha-\epsilon(\alpha-\beta)}$. 

To estimate the contribution of those $\mathbf z\in \partial \mathfrak S_N$ with $z_3=0$ and $z_4\ge N/3$ 
we use Theorem. \ref{thm-second} to find
$$\mathbf P(\mathbf  X_\tau =\mathbf z  | \mathbf X_0=\mathbf s) \asymp 
\frac{N^{\epsilon(\alpha-\beta)} (z_1+z_2)^{\alpha-2\beta+1}(z_1z_2)^{\beta-1}}{d^{1+2\beta }(N^\epsilon+d)^{2(\alpha-\beta)}}
.$$ To sum the formula above over all $(z_1,z_2,0,z_4)$ in $\mathfrak T_{N,3}$ with $z_4\ge N/3$, we consider two cases. When  $\max\{z_1,z_2\}\le N^\epsilon$, 
which implies $N-z_4\le 2 N^\epsilon$ we have  $d\asymp N^\epsilon$. Otherwise, 
$\max\{z_1,z_2\}\ge N^\epsilon$, and  $d\asymp ( z_1+z_2)\ge N^\epsilon$. Call the corresponding regions $U_1$ and $U_2$, respectively. In the first case
$$\sum_{U_1}\mathbf P(\mathbf  X_\tau =\mathbf z  | \mathbf X_0=\mathbf s) \asymp 
N^{-\epsilon\beta}
.$$ In the second case (using symmetry, we can assume $d\asymp z_2\ge z_1$)
\begin{eqnarray*}\sum_{U_1}\mathbf P(\mathbf  X_\tau =\mathbf z  | \mathbf X_0=\mathbf s) & \asymp & 
N^{\epsilon(\alpha-\beta)}\sum_{ N^\epsilon\le d\le  N} \left(d^{-1-\alpha-\beta} \sum_1^dz_1^{\beta-1}\right)\\
&\asymp & N^{\epsilon(\alpha-\beta)}
 \sum_{ N^\epsilon\le d\le  N} d^{-1-\alpha}\asymp N^{-\epsilon \beta}
\end{eqnarray*}
It follows that, for $\mathbf s=(1,1,s_3,N-s_3)$ and $s_3\asymp N^\epsilon$, we have
$$\mathbf P(\mathbf  X_\tau \in \mathfrak T_{N,3}  | \mathbf X_0=\mathbf s)\asymp N^{-\epsilon \beta}.$$

It is useful to compare this with the following heuristic: If player $4$ has $x_4\sim N$ chips, player $3$ $x_3\sim N^\epsilon$
chips and players $2$ and $1$ have just one chip each, to estimate the probability that player $3$ looses, we can ignore player $4$ and imagine  we are watching a game with $3$ players, $N^\epsilon$ chips, and with player $2$ and $1$ having only one chip each. From the $3$-player game results, the probability that
 the very dominant player $3$ looses such a game is of order $N^{-3\epsilon}$. Hence, this heuristic is off since the correct order of magnitude is $N^{-\epsilon \beta}$ and $\beta\approx 2.55$
\end{exa}

\bibliographystyle{plain}

\begin{thebibliography}{10}

\bibitem{Cover1987}
Thomas~M. Cover.
\newblock Gambler's ruin: A random walk on the simplex.
\newblock In Thomas~M. Cover and B.~Gopinath, editors, {\em Open Problems in
  Communication and Computation}, pages 155--155. Springer New York, New York,
  NY, 1987.

\bibitem{DW1}
Denis Denisov and Vitali Wachtel.
\newblock Random walks in cones.
\newblock {\em Ann. Probab.}, 43(3):992--1044, 2015.

\bibitem{DW2}
Denis Denisov and Vitali Wachtel.
\newblock Alternative constructions of a harmonic function for a random walk in
  a cone.
\newblock {\em Electron. J. Probab.}, 24:Paper No. 92, 26, 2019.
\newblock Author name corrected by publisher.

\bibitem{DE21}
Persi Diaconis and Stewart~N. Ethier.
\newblock Gambler's ruin and the icm, 2021.

\bibitem{DHESC1}
Persi Diaconis, Kelsey Houston-Edwards, and Laurent Saloff-Coste.
\newblock Analytic-geometric methods for finite {M}arkov chains with
  applications to quasi-stationarity.
\newblock {\em ALEA Lat. Am. J. Probab. Math. Stat.}, 17(2):901--991, 2020.

\bibitem{DHESC2}
Persi Diaconis, Kelsey Houston-Edwards, and Laurent Saloff-Coste.
\newblock Gambler's ruin estimates on finite inner uniform domains.
\newblock {\em Ann. Appl. Probab.}, 31(2):865--895, 2021.

\bibitem{LierlH}
Janna Lierl.
\newblock Parabolic {H}arnack inequality for time-dependent non-symmetric
  {D}irichlet forms.
\newblock {\em J. Math. Pures Appl. (9)}, 140:1--66, 2020.

\bibitem{LierlSC}
Janna Lierl and Laurent Saloff-Coste.
\newblock The {D}irichlet heat kernel in inner uniform domains: local results,
  compact domains and non-symmetric forms.
\newblock {\em J. Funct. Anal.}, 266(7):4189--4235, 2014.

\bibitem{SWKF}
Varun Shankar, Grady~B. Wright, Robert~M. Kirby, and Aaron~L. Fogelson.
\newblock A radial basis function ({RBF})-finite difference ({FD}) method for
  diffusion and reaction-diffusion equations on surfaces.
\newblock {\em J. Sci. Comput.}, 63(3):745--768, 2015.

\bibitem{Var}
N.~Th. Varopoulos.
\newblock The central limit theorem in {L}ipschitz domains.
\newblock {\em Boll. Unione Mat. Ital.}, 7(2):103--156, 2014.

\bibitem{GW}
Grady Wright.
\newblock private communication, Sep. 2021.

\end{thebibliography}

\end{document}